\documentclass{amsart}
\usepackage{amsmath,amssymb,tikz,enumerate}
\usetikzlibrary{arrows,shapes,positioning,decorations.markings}
\usepackage[margin=1.5cm]{geometry}
\usepackage[colorlinks=true,linkcolor=green]{hyperref}

\makeatletter
\renewcommand*{\eqref}[1]{%
  \hyperref[{#1}]{\textup{\tagform@{\ref*{#1}}}}%
}
\makeatother
\newtheorem{theorem}{Theorem}[section]
\newtheorem{lemma}[theorem]{Lemma}

\newtheorem{conjecture}[theorem]{Conjecture}
\newtheorem{problem}[theorem]{Problem}

\newtheorem{definition}[theorem]{Definition}

\newcommand{\Cay}{\mathop{\Gamma}}
\newcommand{\Aut}{\mathop{\mathrm{Aut}}}

\def\Dic{{\rm Dic}}
\makeatletter
\newcommand{\myitem}[1]{%
\item[#1]\protected@edef\@currentlabel{#1}%
}
\makeatother

\begin{document}

\title[Many suborbits of size at most 2]{Finite transitive groups having many suborbits of cardinality at most two and an application to the enumeration of Cayley graphs}

\author[P. Spiga]{Pablo Spiga}
\address{Pablo Spiga,
Dipartimento di Matematica e Applicazioni, University of Milano-Bicocca,\newline
Via Cozzi 55, 20125 Milano, Italy}\email{pablo.spiga@unimib.it}

\thanks{Address correspondence to P. Spiga,
E-mail: pablo.spiga@unimib.it.}

\begin{abstract}
Let $G$ be a finite transitive group on a set $\Omega$, let $\alpha\in \Omega$ and let $G_\alpha$ be the stabilizer of the point $\alpha$ in $G$. In this paper, we are interested in the proportion $$\frac{|\{\omega\in \Omega\mid \omega \textrm{ lies in a }G_\alpha\textrm{-orbit of cardinality at most two}\}|}{|\Omega|},$$ that is, the proportion of elements of $\Omega$ lying in a suborbit of cardinality at most two. We show that, if this proportion is greater than $5/6$, then each element of $\Omega$ lies in a suborbit of cardinality at most two and hence $G$ is classified by a result of Bergman and Lenstra. We also classify the permutation groups attaining the bound $5/6$. 

We use these results to answer a question concerning the enumeration of Cayley graphs. Given a transitive group $G$ containing a regular subgroup $R$, we determine an upper bound on the number of Cayley graphs on $R$ containing $G$ in their automorphism groups.
\end{abstract}

\maketitle

\section{Introduction}\label{intro}
This paper is part of a series~\cite{MMS,MSMS,Bxu} aiming to obtain an asymptotic enumeration of finite Cayley graphs. However, the main players in this paper are not finite Cayley graphs, but finite transitive groups. Our results on finite transitive groups can then be used to make a considerable step towards the enumeration problem of Cayley graphs and thus getting closer to solving an outstanding question of Babai and Godsil, see~\cite{BaGo} or~\cite[Conjecture~3.13]{Go2}.

Let $G$ be a finite transitive group on $\Omega$, let $\alpha\in \Omega$ and let $G_\alpha$ be the stabilizer in $G$ of the point $\alpha$. The orbits of $G_\alpha$ on $\Omega$ are said to be the \textbf{\textit{suborbits}} of $G$ and their cardinalities are said to be the \textbf{\textit{subdegrees}} of $G$. In this paper, we are concerned in finite transitive groups having many subdegrees equal to $1$ or $2$. In particular, we are interested in the ratio
 $${\bf I}_\Omega(G):=\frac{|\{\omega\in \Omega\mid \omega \textrm{ lies in a }G_\alpha\textrm{-orbit of cardinality at most two}\}|}{|\Omega|}.$$
As $G$ is transitive on $\Omega$, the value of ${\bf I}_\Omega(G)$ does not depend on $\alpha$. Clearly, $0<{\bf I}_\Omega(G)\le 1$.
\begin{theorem}\label{thrm:main1}
Let $G$ be a finite transitive group on $\Omega$, let $\alpha\in \Omega$ and let $G_\alpha$ be the stabilizer in $G$ of the point $\alpha$. If 
${\bf I}_\Omega(G)>\frac{5}{6},$
then ${\bf I}_G(G_\alpha)=1$, that is, each suborbit of $G$ has cardinality at most $2$.
\end{theorem}

It turns out that finite transitive groups $G$ with ${\bf I}_\Omega(G)=1$ are classified by a classical result of Bergman and Lenstra~\cite{BL}. The result of Bergman and Lenstra is rather general and applies to arbitrary  (i.~e.~not necessarily finite) groups. The proof of~\cite[Theorem~1]{BL} is very beautiful and it is based on certain equivalence relations; also the strengthening of Isaacs~\cite{isaacs} of the theorem of Bergman and Lenstra has a remarkably ingenious proof.

 From~\cite[Theorem~1]{BL}, finite transitive groups with ${\bf I}_\Omega(G)=1$ can be partitioned in three families
\begin{enumerate}[(a)]
\item finite transitive groups $G$ where the stabilizer $G_\alpha$ has order $1$,
\item  finite transitive groups $G$ where the stabilizer $G_\alpha$ has order $2$,
\item  finite transitive groups $G$ admitting an elementary abelian normal $2$-subgroup $N$  with $|N:G_\alpha|=2$.
\end{enumerate}
In the first family, each suborbit of $G$ has cardinality $1$, that is, $G$ acts regularly on $\Omega$. In the second family, since $G_\alpha$ has cardinality $2$, each orbit of $G_\alpha$ has cardinality at most $2$. In the third family, since $N\unlhd G$, the orbits of $N$ on $\Omega$ form a system of imprimitivity for the action of $G$; as $|N:G_\alpha|=2$, the blocks of this system of imprimitivity have cardinality $2$ and hence all orbits of $G_\alpha$ have cardinality at most $2$.

Theorem~\ref{thrm:main1} shows that, with respect to the operator ${\bf I}_\Omega(G)$, there is a gap between $5/6$ and $1$. The value $5/6$ is special: there exist finite transitive groups attaining the value $5/6$. 

\begin{theorem}\label{thrm:main2}
Let $G$ be a finite transitive group on $\Omega$, let $\alpha\in \Omega$ and let $G_\alpha$ be the stabilizer in $G$ of the point $\alpha$. If 
${\bf I}_\Omega(G)=\frac{5}{6},$
then there exists an elementary abelian normal $2$-subgroup $N$ of $G$ with $|V:G_\alpha|=|G_\alpha|=4$. 

Moreover, let $e_1,e_2,e_3,e_4$ be a basis of $V$, regarded as a $4$-dimensional vector space over the field with $2$ elements, with $G_\alpha=\langle e_1,e_2\rangle$, let $H:=G/{\bf C}_G(V)$ where ${\bf C}_G(V)$ is the centralizer of $V$ in $G$, and let $K$ be the stabilizer of the subspace $W$ in $\mathrm{GL}(V)$. Then, $H$ is $K$-conjugate to one of the following two groups:
$$
\left\langle
\begin{pmatrix}
0& 0& 0& 1\\
1& 1& 0& 0\\
0& 0& 1& 0\\
1& 0& 0& 1
\end{pmatrix},
\begin{pmatrix}
1& 1& 1& 1\\
0& 0& 1& 0\\
0& 1& 0& 0\\
0& 0& 0& 1
\end{pmatrix}
\right\rangle,\,\,\,
\left\langle
\begin{pmatrix}
0& 0& 0& 1\\
1& 1& 0& 0\\
0& 0& 1& 0\\
1& 0& 0& 1
\end{pmatrix},
\begin{pmatrix}
1& 1& 1& 1\\
0& 0& 1& 0\\
0& 1& 0& 0\\
0& 0& 0& 1
\end{pmatrix},
\begin{pmatrix}
1& 0& 0& 0\\
0& 1& 0& 0\\
1& 1& 0& 1\\
1& 1& 1& 0
\end{pmatrix}
\right\rangle.
$$
The first group has order $12$ and is isomorphic to the alternating group of degree $4$ and the second group has order $24$ and is isomorphic to the symmetric group of degree $4$.

Conversely, if $G$ is a finite group containing an elementary abelian normal $2$-subgroup $V:=\langle e_1,e_2,e_3,e_4\rangle$ of order $16$ and $H:=G/{\bf C}_G(V)$ is as above, then the action of $G$ on the set $\Omega$ of the right cosets of $\langle e_1,e_2\rangle$ gives rise to a finite permutation group of degree $4|G:V|$ with ${\bf I}_\Omega(G)=5/6$.
\end{theorem}
Theorem~\ref{thrm:main2} classifies the finite transitive groups attaining the bound $5/6$.

Before discussing our motivation for proving Theorems~\ref{thrm:main1} and~\ref{thrm:main2}, we make some speculations. A computer search among the transitive groups $G$ of degree at most $48$  with the computer algebra system \texttt{magma}~\cite{magma} reveals that, if ${\bf I}_\Omega(G)>1/2$, then ${\bf I}_\Omega(G)=(q+1)/2q$, for some $q\in\mathbb{Q}$ with $2q\in\mathbb{N}$. We pose this as a conjecture.
\begin{conjecture}\label{conj}
{\rm Let $G$ be a finite transitive group on $\Omega$. If  ${\bf I}_\Omega(G)>1/2$, then ${\bf I}_\Omega(G)=(q+1)/2q$, for some $q\in\mathbb{Q}$ with $2q\in\mathbb{N}$.}
\end{conjecture}
If true, Conjecture~\ref{conj} establishes a permutation analogue with a classical problem in finite group theory. Let $G$ be a finite group and let $${\bf I}(G):=\{x\in G\mid x \textrm{ has order at most 2}\}.$$ Miller~\cite{Miller} in 1905 has shown that, if ${\bf I}(G)>3/4$, then each element of $G$ has order at most $2$ and hence $G$ is an elementary abelian $2$-group. In this regard, Theorem~\ref{thrm:main1} can be seen as a permutation analogue of the theorem of Miller, with the only difference that the ratio $3/4$ in the context of abstract groups has to bump up to $5/6$ in the context of permutation groups. Miller has also classified the finite groups $G$ with ${\bf I}(G)=3/4$. Therefore, Theorem~\ref{thrm:main2} can be seen as a permutation analogue of the classification of Miller. The theorem of Miller has stimulated a lot of research; for instance, Wall~\cite{Wall} has classified all finite groups $G$ with ${\bf I}(G)>1/2$. In his proof, Wall uses the Frobenius-Schur formula for counting involutions. An application of this classification shows that, if ${\bf I}(G)>1/2$, then ${\bf I}(G)=(q+1)/2q$, for some positive integer $q$. Therefore, in Conjecture~\ref{conj}, we believe that the same type of result holds for the permutation analogue ${\bf I}_\Omega(G)$, but allowing $q$ to be an element of $\{x/2\mid x\in\mathbb{N}\}$. As a wishful thinking, we also pose the following problem.

\begin{problem}\label{problema}
{\rm 
Classify the  finite transitive groups $G$ acting on $\Omega$ with ${\bf I}_\Omega(G)>1/2$.}
\end{problem}

Liebeck and MacHale~\cite{LieMac} have generalized the results of Miller and Wall in yet another direction.  Indeed, Liebeck and MacHale have classified the  finite groups $G$ admitting an automorphism inverting more than half of the elements of $G$. (The classical results of Miller and Wall can be recovered by considering the identity automorphism.) Then,  this classification has been pushed even further by Fitzpatrick~\cite{fitzpatrick} and Hegarty and MacHale~\cite{hegarty}, by classifying the finite groups $G$ admitting an automorphism inverting exactly half of the elements of $G$. An application of this classification shows that, if $\alpha$ is an automorphism of $G$ inverting more than half of  the elements of $G$, then the proportion of elements inverted by $\alpha$ is $(q+1)/2q$, for some positive integer $q$. Yet again, another analogue with Theorems~\ref{thrm:main1} and~\ref{thrm:main2}, with Conjecture~\ref{conj} and with Problem~\ref{problema}. We observe that a partial generalization of this type of results in the context of association schemes is in~\cite{MZ}.

We now discuss our original motivation for proving Theorems~\ref{thrm:main1} and~\ref{thrm:main2}. A \textbf{\textit{digraph}} $\Gamma$ is an ordered pair $(V,E)$ with $V$ a finite non-empty set of vertices, and $E$ is a subset of $V\times V$, representing the arcs. A \textbf{\textit{graph}} $\Gamma$ is a digraph $(V,E)$, where the binary relation $E$ is symmetric. An automorphism of a (di)graph is a permutation on $V$ that preserves the set $E$.

\begin{definition}{\rm 
Let $R$ be a group and let $S$ be a subset of $R$.  The \textbf{\emph{Cayley digraph}} $\Cay(R,S)$ is the digraph with $V=R$ and $(r,t) \in E$ if and only if $tr^{-1} \in S$.

The Cayley digraph is a graph if and only if $S=S^{-1}$, that is, $S$ is an inverse-closed subset of $R$.} 
\end{definition}

The problem of finding graphical regular representations (GRRs) for groups has a long history. Mathematicians have studied graphs with specified automorphism groups at least as far back as the 1930s, and in the 1970s there were many papers devoted to the topic of finding GRRs (see for example \cite{babai11,Het,Im1, Im2,Im3,NW1,NW2,NW3,Wat}), although the ``GRR" terminology was coined somewhat later.

\begin{definition}{\rm
A \textbf{\emph{digraphical regular representation}} (DRR) for a group $R$ is a digraph whose full automorphism group is the group $R$ acting regularly on the vertices of the digraph.

Similarly, a \textbf{\emph{graphical regular representation}} (GRR) for a group $R$ is a graph whose full automorphism group is the group $R$ acting regularly on the vertices of the graph.}
\end{definition}

It is an easy observation that when $\Cay(R,S)$ is a Cayley (di)graph, the group $R$ acts regularly on the vertices as a group of graph automorphisms. A DRR (or GRR) for $R$ is therefore a Cayley (di)graph on $R$ that admits no other automorphisms.

The main thrust of much of the work through the 1970s was to determine which groups admit GRRs. This question was ultimately answered by Godsil in~\cite{God}. The  corresponding result for DRRs was proved by a much simpler argument by Babai~\cite{babai11}.

Babai and Godsil made the following conjecture. (Given a finite group $R$,  $2^{{\bf c}(R)}$ denotes the number of inverse-closed subsets of $R$. See Definition~\ref{defeq:2} for the definition of generalized dicyclic group.)
\begin{conjecture}[\cite{BaGo}; Conjecture 3.13, \cite{Go2}]
{\rm
If $R$ is not generalised dicyclic or abelian of exponent greater than $2$, then for almost all inverse-closed subsets $S$ of $R$, $\Cay(R,S)$ is a GRR. In other words,
$$\lim_{|R| \to \infty} \min\left\{ \frac{|\{S \subseteq R: S=S^{-1},\,\Aut(\Cay(R,S))=R\}|}{2^{{\bf c}(R)}}: R\text{ admits a GRR}|\right\} =1.$$}
\end{conjecture}

From Godsil's theorem~\cite{God}, as $|R|\to \infty$, the condition ``$R$ admits a GRR" is equivalent to ``$R$ is neither a generalised dicyclic group, nor abelian of exponent greater than $2$."

The corresponding conjecture for Cayley digraphs (which does not  require any families of groups to be excluded) was proved by Morris and the author in~\cite{MSMS}. Our current strategy for proving the conjecture of Babai and Godsil is to use the proof of the corresponding conjecture for Cayley digraphs as a template and extend the work in~\cite{MSMS} in the context of undirected Cayley graphs. This strategy so far has been rather successful and in~\cite{MMS,Bxu} the authors have already adapted some of the arguments in~\cite{MSMS} for undirected graphs. 

One key tool in~\cite{MSMS} is an elementary observation of Babai. 
\begin{lemma}\label{lemma1}
Let $G$ be a finite transitive group properly containing a regular subgroup $R$. Then there are at most $2^{\frac{3|\Omega|}{4}}$ Cayley digraphs $\Gamma$ on $R$ with  $G\leq \Aut(\Gamma)$. 
\end{lemma}
The proof of this fact is elementary, see for instance~\cite[Lemma~1.8]{MSMS}. Observe that the number of Cayley digraphs on $R$ is the number of subsets of $R$, that is, $2^{|R|}$. Therefore, Lemma~\ref{lemma1} says that, given $G$ properly containing $R$, only at most $2^{|R|-\frac{|R|}{4}}$ of these Cayley digraphs admit $G$ as a group of automorphisms. This gain of $|R|/4$ is one of the tools in~\cite{MSMS} for proving the Babai-Godsil conjecture on Cayley digraphs.

 To continue our project of proving the Babai-Godsil conjecture for Cayley graphs, we need an analogue of Lemma~\ref{lemma1} for Cayley graphs. Observe that the number of Cayley graphs on $R$ is the number of inverse-closed subsets of $R$. We denote this number with $2^{{\bf c}(R)}$. It is not hard to prove (see for instance~\cite[Lemma~$1.12$]{MMS}) that $${\bf c}(R)=\frac{|R|+|{\bf I}(R)|}{2},$$
 where ${\bf I}(R)=\{x\in R\mid x^2=1\}$.
 To obtain this analogue one needs to investigate finite transitive groups having many suborbits of cardinality at most $2$. Therefore, our investigation leads to the following result.
 
\begin{theorem}\label{thrm:main3}
Let $G$ be a finite transitive group properly containing a regular subgroup $R$. Then one of the following holds
\begin{enumerate}[(a)]
\item the number  of Cayley graphs $\Gamma$ on $R$ with $G\leq \Aut(\Gamma)$ is at most $2^{{\bf c}(R)-\frac{|R|}{96}}$, 
\item $R$ is abelian of exponent greater than $2$,
\item $R$ is generalized dicyclic (see Definition~$\ref{defeq:2}$).
\end{enumerate}
\end{theorem}

\subsection{Notation}\label{notation}
In this section, we establish some notation that we use throughout the rest of the paper. 

Given a subset $X$ of permutations from $\Omega$, we use an exponential notation for the action on $\Omega$ and hence, in particular, given $\omega\in \Omega$, we let
$$\omega^X:=\{\omega^x\mid x\in X\},$$
where $\omega^x$ is the image of $\omega$ under the permutation $x$.
Similarly, we let
$$\mathrm{Fix}_\Omega(X):=\{\omega\in \Omega\mid \omega^x=\omega,\,\forall x\in X\}.$$

Let $G$ be a transitive permutation group on $\Omega$. 
For each positive integer $i$ and for each $\omega\in \Omega$, we let 
\begin{equation}\label{notation:1}
\Omega_{\omega,i}:=\{\delta\in \Omega\mid |\delta^{G_\omega}|=i\}.
\end{equation}

Clearly, 
\begin{equation}\label{eq:omega0}
\Omega=\Omega_{\omega,1}\cup\Omega_{\omega,2}\cup\Omega_{\omega,3}\cup\cdots
\end{equation} and the non-empty sets in this union form a partition of $\Omega$.

When $i:=1$, we have $$\Omega_{\omega,1}=\{\delta\in \Omega\mid G_\omega\textrm{ fixes }\delta\},$$ that is, $\Omega_{\omega,1}$ is the set of fixed points of $G_\omega$ on $\Omega$. It is well-known that $\Omega_{\omega,1}$ is a block of imprimitivity for the action of $G$ on $\Omega$, see for instance~\cite[1.6.5]{dixonmortimer}. Since this fact will play a role in what follows, we prove it here; this will also be helpful for setting up some additional notation. Let ${\bf N}_{G}(G_\omega)$ be the normalizer of $G_\omega$  in $G$. As ${\bf N}_G(G_\omega)$ contains $G_\omega$, the ${\bf N}_G(G_\omega)$-orbit containing $\omega$ is a block of imprimitivity for the action of $G$ on $\Omega$. Therefore it suffices to prove that $\Omega_{\omega,1}$ is the ${\bf N}_G(G_\omega)$-orbit containg $\omega$, that is, $\Omega_{\omega,1}=\omega^{{\bf N}_G(G_\omega)}=\{\omega^g\mid g\in {\bf N}_G(G_\omega)\}$. If $g\in {\bf N}_G(G_\omega)$, then $G_\omega=G_\omega^g=G_{\omega^g}$ and hence $G_\omega$ fixes $\omega^g$, that is, $\omega^g\in \Omega_{\omega,1}$. Conversely, let $\alpha\in \Omega_{\omega,1}$. As $G$ is transitive on $\Omega$, there exists $g\in G$ with $\alpha=\omega^g$. Thus $\omega^g\in \Omega_{\omega,1}$ and $G_\omega$ fixes $\omega^g$. This yields $G_\omega=G_{\omega^g}=G_\omega^g$ and $g\in{\bf N}_G(G_\omega)$. Therefore, $\alpha=\omega^g$ lies in the ${\bf N}_G(G_\omega)$-orbit containg $\omega$.

We let
\begin{equation*}%\label{def:di}
d:=|\Omega_{\omega,1}|.
\end{equation*}
As $G$ is transitive on $\Omega$, $d$ does not depend on $\omega$. From the previous paragraph, we deduce that $d$ divides $|\Omega_{\omega,i}|$, for each positive integer $i$. We define
\begin{equation}\label{def:xi}
x_i:=\frac{|\Omega_{\omega,i}|}{|\Omega_{\omega,1}|}=\frac{|\Omega_{\omega,i}|}{d}\in\mathbb{N}.
\end{equation}
In particular, $x_1:=1$ and, from~\eqref{eq:omega0} and~\eqref{def:xi}, we have 
\begin{equation*}%\label{eq:omega}
|\Omega|=d\sum_{i}x_i.
\end{equation*}

\begin{definition}\label{defeq:2}{\rm
Let $A$ be an abelian group of even order and of exponent greater than $2$, and let $y$ be an involution of $A$. The generalised dicyclic group $\Dic(A, y, x)$ is the group $\langle A, x\mid x^2=y, a^x=a^{-1},\forall a\in A\rangle$. A group is called \textit{\textbf{generalised dicyclic}} if it is isomorphic to some $\Dic(A, y, x)$. When $A$ is cyclic, $\Dic(A, y, x)$ is called a dicyclic or generalised quaternion group.
}
\end{definition}

\section{Lemmata}\label{sec:lemmata}
In this section we use the notation established in Section~\ref{notation}.
\begin{lemma}\label{lemma:-4}
Let $G$ be a finite permutation group on a set $\Omega$ and let $\alpha\in \Omega$. If
$$\frac{|\Omega|}{2}<|\Omega_{\alpha,1}|+|\Omega_{\alpha,2}|<|\Omega|,$$
then 
\begin{enumerate}[(a)]
\item\label{eq:lemma-41}$\Omega=\Omega_{\alpha,1}\cup\Omega_{1,2}\cup\Omega_{\alpha,4}$ (in particular, $\Omega_{\alpha,i}=\emptyset$, for every positive integer $i$ with $i\notin\{1,2,4\}$);
\item\label{eq:lemma-42} for every $\beta\in \Omega_{\alpha,4}$, $\Omega_{\alpha,2}\cap\Omega_{\beta,2}\ne\emptyset$;
\item\label{eq:lemma-43}for every $\beta\in\Omega_{\alpha,4}$ and for every $\omega\in \Omega_{\alpha,2}\cap\Omega_{\beta,2}$, we have
$G_\omega=(G_\alpha\cap G_\omega)(G_\beta\cap G_\omega)$.
\end{enumerate}
\end{lemma}
\begin{proof}
As $|\Omega_{\alpha,1}|+|\Omega_{\alpha,2}|<|\Omega|$, by~\eqref{eq:omega0}, we get that $\Omega_{\alpha,1}\cup\Omega_{\alpha,2}$ is strictly contained in $\Omega$. Therefore, 
let $\beta\in \Omega\setminus(\Omega_{\alpha,1}\cup\Omega_{\alpha,2})$. 

Since $\beta\notin \Omega_{\alpha,1}\cup\Omega_{\alpha,2}$, we have  
\begin{equation}\label{eye-1}
|G_\alpha:G_\alpha\cap G_\beta|=|G_\beta:G_\alpha\cap G_\beta|>2.
\end{equation} See Figure~\ref{figureeye0}.
\begin{figure}[!ht]
\begin{tikzpicture}[node distance=1.3cm]
\node at (0,0) (A0) {$G_\alpha$};
\node[right of=A0] (A1) {};
\node[right of=A1] (A2) {$G_\beta$};
\node[below of=A1] (A3) {$G_\alpha\cap G_\beta$};
\draw[-] (A0) -- node[left]{$>2$}(A3);
\draw[-] (A2) -- node[right]{$>2$}(A3);
\end{tikzpicture}
\caption{}\label{figureeye0}
\end{figure} 

From this we deduce
\begin{equation}\label{eye:2}\Omega_{\alpha,1}\cap \Omega_{\beta,1}=\Omega_{\alpha,2}\cap \Omega_{\beta,1}=\Omega_{\alpha,1}\cap \Omega_{\beta,2}=\emptyset.
\end{equation}
Indeed, if for instance $\omega\in \Omega_{\alpha,1}\cap\Omega_{\beta,2}$, then $|\omega^{G_\alpha}|=1$ and $|\omega^{G_\beta}|=2$. Therefore, $|G_\alpha:G_\alpha\cap G_\omega|=1$ and $|G_\beta:G_\beta\cap G_\omega|=2$. As $|G_\alpha:G_\alpha\cap G_\omega|=1$, we get $G_\alpha=G_\omega$. Now, as $|G_\beta:G_\beta\cap G_\omega|=2$ and $G_\alpha=G_\omega$, we get $2=|G_\beta:G_\beta\cap G_\omega|=|G_\beta:G_\beta\cap G_\alpha|$, which contradicts~\eqref{eye-1}. Therefore $\Omega_{\alpha,1}\cap \Omega_{\beta,2}=\emptyset$. The proof for all other equalities in~\eqref{eye:2} is similar.

From~\eqref{eye:2}, we obtain
\begin{equation}\label{eye:3}
(\Omega_{\alpha,1}\cup\Omega_{\alpha,2})\cap (\Omega_{\beta,1}\cup\Omega_{\beta,2})=\Omega_{\alpha,2}\cap\Omega_{\beta,2}.
\end{equation}
Recall that, by hypothesis, $|\Omega_{\alpha,1}\cup\Omega_{\alpha,2}|>|\Omega|/2$. Using this together with~\eqref{eye:3}, we get
\begin{align}\label{eye:4}
|\Omega_{\alpha,2}\cap\Omega_{\beta,2}|&=|(\Omega_{\alpha,1}\cup\Omega_{\alpha,2})\cap (\Omega_{\beta,1}\cup\Omega_{\beta,2})|\\\nonumber
&=|\Omega_{\alpha,1}\cup\Omega_{\alpha,2}|+ |\Omega_{\beta,1}\cup\Omega_{\beta,2}|-|(\Omega_{\alpha,1}\cup\Omega_{\alpha,2})\cup (\Omega_{\beta,1}\cup\Omega_{\beta,2})|\\\nonumber
&\ge |\Omega_{\alpha,1}\cup\Omega_{\alpha,2}|+ |\Omega_{\beta,1}\cup\Omega_{\beta,2}|-|\Omega|\\\nonumber
&>\frac{|\Omega|}{2}+\frac{|\Omega|}{2}-|\Omega|=0.
\end{align}

From~\eqref{eye:4}, we deduce $\Omega_{\alpha,2}\cap\Omega_{\beta,2}\ne\emptyset$. Let $\omega\in \Omega_{\alpha,2}\cap \Omega_{\beta,2}$. In particular, $|\omega^{G_\alpha}|=|\omega^{G_\beta}|=2$. This means that $|G_\alpha:G_\alpha\cap G_\omega|=|G_\beta:G_\beta\cap G_\omega|=2$. Since $|G_\alpha|=|G_\beta|=|G_\omega|$, we get that $G_\alpha\cap G_\omega$ and $G_\beta\cap G_\omega$ have both index $2$ in $G_\omega$. Suppose $G_\alpha\cap G_\omega=G_\beta\cap G_\omega$. Then
$$G_\alpha\cap G_\omega=G_\beta\cap G_\omega=G_\alpha\cap G_\beta\cap G_\omega\le G_\alpha\cap G_\beta$$
and hence
$$|G_\alpha:G_\alpha\cap G_\beta|\le |G_\alpha:G_\alpha\cap G_\omega|=2.$$
However, this contradicts~\eqref{eye-1}.  Therefore, $G_\alpha\cap G_\omega$ and $G_\beta\cap G_\omega$ are two distinct subgroups of $G_\omega$ having index $2$. This yields
\begin{equation}\label{eye:5}G_\omega=(G_\alpha\cap G_\omega)(G_\beta\cap G_\omega),
\end{equation}
for each $\omega\in \Omega_{\alpha,2}\cap\Omega_{\beta,2}$.

From~\eqref{eye:5} and from the fact that $|G_\omega:G_\alpha\cap G_\omega|=|G_\omega:G_\beta\cap G_\omega|=2$, we see that $(G_\alpha\cap G_\omega)\cap (G_\beta\cap G_\omega)=G_\alpha\cap G_\beta\cap G_\omega$ has index $4$ in $G_\omega$. Since $|G_\omega|=|G_\alpha|=|G_\beta|$, we get that $G_\alpha\cap G_\beta\cap G_\omega$ has also index $4$ in $G_\alpha$ and in $G_\beta$.
Since $G_\alpha\cap G_\beta\cap G_\omega\le G_\alpha\cap G_\beta$, we get that $|G_\alpha:G_\alpha\cap G_\beta|=|G_\beta:G_\alpha\cap G_\beta|$ divides $|G_\alpha:G_\alpha\cap G_\beta\cap G_\omega|=4$. As $|G_\alpha:G_\alpha\cap G_\beta|=|G_\beta:G_\alpha\cap G_\beta|>2$, we get 
$G_\alpha\cap G_\beta\cap G_\omega= G_\alpha\cap G_\beta$ 
and
 $$|G_\alpha:G_\alpha\cap G_\beta|=|G_\beta:G_\alpha\cap G_\beta|=4.$$ We have summarized this paragraph in Figure~\ref{figureeye1}.  In other words, $\beta\in \Omega_{\alpha,4}$.
\begin{figure}[!ht]
\begin{tikzpicture}[node distance=2.5cm]
\node at (0,0) (A0) {$G_\omega=(G_\alpha\cap G_\omega)(G_\beta\cap G_\omega)$};
\node[left of=A0] (A1) {};
\node[left of=A1](AA1){$G_\alpha$};
\node[right of=A0] (A2) {};
\node[right of=A2](AA2){$G_\beta$};
\node[below of=A1] (A3) {$G_\alpha\cap G_\omega$};
\node[below of=A2] (A4){$G_\beta\cap G_\omega$};
\node[below of=A0] (B){};
\node[below of=B] (A5){$G_\alpha\cap G_\beta$};
\draw[-] (A0) -- node[left]{2}(A3);
\draw[-] (AA1) -- node[left]{2}(A3);
\draw[-] (A0) -- node[right]{2}(A4);
\draw[-] (AA2) -- node[right]{2}(A4);
\draw[-] (A3) -- node[left]{2}(A5);
\draw[-] (A4) -- node[right]{2}(A5);
\end{tikzpicture}
\caption{}\label{figureeye1}
\end{figure}

Since $\beta$ is an arbitrary element in $\Omega\setminus(\Omega_{\alpha,1}\cup\Omega_{\alpha,2})$, we have proven part~\eqref{eq:lemma-41}. Now, as $\Omega_{\alpha,4}=\Omega\setminus(\Omega_{\alpha,1}\cup\Omega_{\alpha,2})$, part~\eqref{eq:lemma-42} follows from~\eqref{eye:4} and part~\eqref{eq:lemma-43} follows from~\eqref{eye:5}. 
\end{proof}

\begin{lemma}\label{lemma:4esteso}
Let $G$ be a finite permutation group on a set $\Omega$ and let $\alpha\in \Omega$. If $\Omega=\Omega_{\alpha,1}\cup\Omega_{\alpha,2}\cup\Omega_{\alpha,4}$ and $|\Omega_{\alpha,1}|=|\Omega_{\alpha,4}|$, then $\Omega_{\alpha,1}\cup\Omega_{\alpha,4}$ is a block of imprimitivity for $G$. Moreover,  ${\bf N}_G(G_\alpha)={\bf N}_G(G_\beta)$.
\end{lemma}
\begin{proof}
Let $\beta\in\Omega_{\alpha,4}$. As $\Omega_{\beta,1}\subseteq \Omega_{\alpha,4}$ and as $\Omega_{\beta,1}$ and $\Omega_{\alpha,4}$ have the same cardinality, we deduce $\Omega_{\alpha,4}=\Omega_{\beta,1}$. Analogously, $\Omega_{\beta,4}=\Omega_{\alpha,1}$.

Let $g\in G$ with $\beta=\alpha^g$. Now, we have
$$(\Omega_{\alpha,4})^g=\Omega_{\alpha^g,4}=\Omega_{\beta,4}=\Omega_{\alpha,1}.$$
Analogously, $\Omega_{\alpha,1}^g=\Omega_{\alpha,4}$. 
So,
$$\Omega_{\alpha,1}^g=\Omega_{\alpha,4}\hbox{ and }\Omega_{\alpha,4}^g=\Omega_{\alpha,1}.$$
Therefore, $(\Omega_{\alpha,1}\cup\Omega_{\alpha,4})^g=\Omega_{\alpha,1}\cup\Omega_{\alpha,4}$ and $g^2$ fixes setwise $\Omega_{\alpha,1}$ and $\Omega_{\alpha,4}$.

Since $\Omega_{\alpha,1}$ is a block of imprimitivity for $G$ with setwise stabilizer ${\bf N}_G(G_\alpha)$, we deduce $g^2\in {\bf N}_G(G_\alpha)$. Set $T:=\langle {\bf N}_G(G_\alpha),g\rangle$.

Since $G_\alpha$ fixes setwise $\Omega_{\alpha,1}\cup\Omega_{\alpha,4}$, we deduce that $G_\alpha$ fixes setwise also $\Omega_{\alpha,4}=\Omega_{\beta,1}$. Now, for every $x\in {\bf N}_G(G_\alpha)$, we have
$$\Omega_{\alpha,1}^{g^{-1}\alpha g}=(\Omega_{\alpha,1}^{g^{-1}})^{xg}=\Omega_{\beta,1}^{xg}=(\Omega_{\beta,1}^x)^g=\Omega_{\beta,1}^g=\Omega_{\alpha,1}.$$
Thus $g^{-1}xg$ fixes setwise $\Omega_{\alpha,1}$ and hence $g^{-1}xg\in {\bf N}_G(G_\alpha)$. This yields $${\bf N}_G(G_\beta)={\bf N}_G(G_{\alpha^g})=({\bf N}_G(G_\alpha))^g={\bf N}_G(G_\alpha).$$

As $g$ normalizes ${\bf N}_G(G_\alpha)$, we have $T={\bf N}_G(G_\alpha)\langle g\rangle$ and
$$\alpha^T=(\alpha^{{\bf N}_G(G_\alpha)})^{\langle g\rangle}=\Omega_{\alpha,1}^{\langle g\rangle}=\Omega_{\alpha,1}\cup\Omega_{\alpha,4}.$$
 Now, since $T$ is an overgroup of $G_\alpha$ and since $\Omega_{\alpha,1}\cup\Omega_{\alpha,4}$ is the $T$-orbit containing $\alpha$, we deduce that $\Omega_{\alpha,1}\cup\Omega_{\alpha,4}$ is a block of imprimitivity for $G$.
\end{proof}

We now need two rather technical lemmas, at first they seem out of context, but their relevance is pivotal in the  proof of Lemma~\ref{lemma:-3}. We could phrase Lemma~\ref{lemma:44esteso} in a purely group theoretic terminology, but it is easier to state in our opinion using some terminology from graph theory.

\begin{lemma}\label{lemma:44esteso}
Let $G$ be a group, let $X$ be an elementary abelian $2$-subgroup of $G$, let $Y$ be a $G$-conjugate of $X$ with $Z:=X\cap Y$ having index $4$ in $X$ and in $Y$. Let $\Lambda_X:=\{X_1,X_2,X_3\}$ and $\Lambda_{Y}:=\{Y_1,Y_2,Y_3\}$ be the collection of the proper subgroups of $X$ and $Y$, respectively, properly containing $Z$.

Let $\Gamma$ be the bipartite graph having vertex set $\Lambda_X\cup\Lambda_Y$, where a pair $\{X_i,Y_j\}$
 is declared to be adjacent if $X_iY_j$ is a subgroup of $G$ conjugate to $X$ via an element of $G$. 
 If $\Gamma$ has at least $6$ edges, then $X$ commutes with $Y$.
\end{lemma}
\begin{proof}
Suppose that 
\begin{center}
$(\ast)\quad$  there exist two distinct vertices of $\Gamma$ having valency at least $2$.
\end{center} By symmetry, without loss of generality, we suppose that these two vertices are in $\Lambda_X$. Thus suppose that $X_i,X_j\in \Lambda_X$  have valency at least $2$ in $\Gamma$.

Let $Y_{i_1}$ and $Y_{i_2}$ be two neighbours of $X_i$ in $\Gamma$. Then, by definition, $X_iY_{i_1}$ and $X_iY_{i_2}$ are both subgroups of $G$ conjugate to $X$. Therefore, $X_iY_{i_1}$ and $X_iY_{i_2}$ are elementary abelian $2$-groups and hence $X_i$ commutes with both $Y_{i_1}$ and $Y_{i_2}$. Since $\langle Y_{i_1},Y_{i_2}\rangle=Y$, we deduce that $X_i$ commutes with $Y$.

Arguing as in the paragraph above with $X_i$ replaced by $X_j$, we deduce that $X_j$ commutes with $Y$. Therefore, $X=\langle X_i,X_j\rangle$ commutes with $Y$.

Now, it is elementary to see that every bipartite graph on six vertices, with parts having cardinality $3$ and having at least $6$ edges has the property $(\ast)$.
\end{proof}

Recally that a graph $\Gamma$ is said to be vertex-transitive if its automorphism group acts transitively on the vertices of $\Gamma$. Given a vertex $\omega$ of $\Gamma$, we denote by $\Gamma(\omega)$ the neighbourhood of $\omega$ in $\Gamma$.
\begin{lemma}\label{lemma:444esteso}
Let $\Gamma$ be a finite vertex-transitive graph having valency $2$, let $V$ be the set of vertices of $\Gamma$, let $\omega_1,\omega_2$ be two adjacent vertices of $\Gamma$ and let $W$ be a subset of $V$ containing $\omega_1$ and $\omega_2$ and with the property that, for any two distinct vertices $\delta_1,\delta_2$ in $W$, $V\setminus (\Gamma(\delta_1)\cup\Gamma(\delta_2))\subseteq W$. Then either $W=V$ or $|V|\le 6$.
\end{lemma}
\begin{proof}Since $\Gamma$ is vertex-transitive of valency $2$, $\Gamma$ is a disjoint union $s$ of cycles of the same length $\ell$. If $\ell\ge 7$ or if $\Gamma$ is disconnected, that is, $s\ge 2$, it can be easily checked that $W=V$.
\end{proof}

\begin{lemma}\label{lemma:-3}
Let $G$ be a finite permutation group on a set $\Omega$ and let $\alpha\in \Omega$. If
$$\frac{|\Omega|}{2}<|\Omega_{\alpha,1}|+|\Omega_{\alpha,2}|<|\Omega|,$$
then one of the following holds
\begin{enumerate}[(a)]
\item\label{lemma:-30}$|\Omega_{\alpha,1}|+|\Omega_{\alpha,2}|< 5|\Omega|/6$, or
\item\label{lemma:-31}
\begin{enumerate}[(i)]
\item\label{lemma:-322}$|\Omega_{\alpha,4}|\le 2|\Omega_{\alpha,1}|$, and
\item\label{lemma:-32}$G_\alpha$ is  an elementary abelian $2$-group, and
\item\label{lemma:-33}$G_\alpha$  commutes with $G_\beta$, for every $\beta\in \Omega_{\alpha,4}$, and
\item\label{lemma:-34}$\langle G_\alpha,G_\beta\rangle=G_\alpha\times G_\beta$ is an elementary abelian normal $2$-subgroup of $G$ of order $16$, for every $\beta\in \Omega_{\alpha,4}$.
\end{enumerate}
\end{enumerate}

\end{lemma}

\begin{proof}
From Lemma~\ref{lemma:-4}, $\Omega=\Omega_{\alpha,1}\cup\Omega_{\alpha,2}\cup \Omega_{\alpha,4}$. Moreover, for each $\beta\in \Omega_{\alpha,4}$, we have shown that $G_\alpha$ contains a proper subgroup (namely, $G_\alpha\cap G_\omega$, for each $\omega\in \Omega_{\alpha,2}\cap\Omega_{\beta,2}$) strictly containg $G_\alpha\cap G_\beta$. This implies that the permutation group, $P$ say, induced by $G_\alpha$ in its action on the suborbit $\beta^{G_\alpha}$ is a $2$-group. (Indeed,  if $G_\alpha$ induces the alternating group $\mathrm{Alt}(4)$ or the symmetric group  $\mathrm{Sym}(4)$ on $\beta^{G_\alpha}$, then $G_\alpha$ acts primitively on $\beta^{G_\alpha}$ and hence $G_\alpha\cap G_\beta$ is maximal in $G_\alpha$.) Clearly, this $2$-group $P$ must be either cyclic of order $4$, or elementary abelian of order $4$, or dihedral of order $8$. 

We have drawn in Figure~\ref{figureeye-2} the lattice of subgroups of the cyclic group of order $4$, the elementary abelian group of order $4$ and the dihedral group of order $8$: the dark colored nodes indicate the lattice of subgroups between the whole group and the stabilizer of a point. 

Figure~\ref{figureeye-2} shows that, given $G_\alpha$ and $G_{\alpha}\cap G_\beta$, we only have one choice for $G_\alpha\cap G_\omega$ when $P$ is cyclic of order $4$ or dihedral of order $8$, whereas we have at most three choices for $G_\alpha\cap G_\omega$ when $P$ is elementary abelian of order $4$.
\begin{figure}[!ht]
\begin{tikzpicture}[node distance=1cm,inner sep=0pt]
\node[minimum size=2mm,circle, fill=black] at (-1,0) (A0) {};
\node[minimum size=2mm,circle,fill=black,below of=A0] (A1) {};
\node[minimum size=2mm,circle,fill=black,below of=A1] (A2) {};
\node[minimum size=2mm,circle, fill=black] at (3,0) (B0) {};
\node[minimum size=2mm,circle,fill=black,below of=B0] (B1) {};
\node[minimum size=2mm,circle,fill=black,below of=B1] (B2) {};
\node[minimum size=2mm,circle,fill=black,left of=B1] (B3) {};
\node[minimum size=2mm,circle,fill=black,right of=B1] (B4) {};
\node[minimum size=2mm,circle, fill=black] at (7,0) (C0) {};
\node[minimum size=2mm,circle,fill=lightgray,below of=C0] (C1) {};
\node[minimum size=2mm,circle,fill=black,left of=C1] (C2) {};
\node[minimum size=2mm,circle,fill=lightgray,right of=C1] (C3) {};
\node[minimum size=2mm,circle,fill=lightgray,below of=C1] (C4) {};
\node[minimum size=2mm,circle,fill=black,left of=C4] (C5) {};
\node[minimum size=2mm,circle,fill=lightgray,left of=C5] (C6) {};
\node[minimum size=2mm,circle,fill=lightgray,right of=C4] (C7) {};
\node[minimum size=2mm,circle,fill=lightgray,right of=C7] (C8) {};
\node[minimum size=2mm,circle,fill=lightgray,below of=C4] (C9) {};

\draw[-] (A0) -- (A1);
\draw[-] (A1) -- (A2);
\draw[-] (B0) -- (B1);
\draw[-] (B0) -- (B3);
\draw[-] (B0) -- (B4);
\draw[-] (B2) -- (B1);
\draw[-] (B2) -- (B3);
\draw[-] (B2) -- (B4);

\draw[-] (C0) -- (C1);
\draw[-] (C0) -- (C2);
\draw[-] (C0) -- (C3);
\draw[-] (C4) -- (C1);
\draw[-] (C4) -- (C2);
\draw[-] (C4) -- (C3);
\draw[-] (C9) -- (C5);
\draw[-] (C9) -- (C6);
\draw[-] (C9) -- (C7);
\draw[-] (C9) -- (C8);
\draw[-] (C9) -- (C4);
\draw[-] (C2) -- (C5);
\draw[-] (C2) -- (C6);
\draw[-] (C3) -- (C7);
\draw[-] (C3) -- (C8);
\end{tikzpicture}
\caption{}\label{figureeye-2}
\end{figure}

%Since the permutation group induced by $G_\alpha$ in the action on each of its orbits is a $2$-group, we deduce that
%\begin{equation*}
%G_\alpha\,\,\textrm{is a }2\textrm{-group.}
%\end{equation*}

Given $\beta\in \Omega_{\alpha,4}$, let $$\mathcal{S}_{\alpha,\beta}:=\{G_\omega\mid \omega\in \Omega_{\alpha,2}\cap \Omega_{\beta,2}\}.$$
Observe that in the set $\mathcal{S}_{\alpha,\beta}$ we are collecting point stabilizers and not elements of $\Omega$ and hence different elements $\omega_1,\omega_2$ of $\Omega$ can give rise to the same element of $\mathcal{S}_{\alpha,\beta}$ when $G_{\omega_1}=G_{\omega_2}$.

We claim that 
\begin{equation}\label{eye:6}
|\mathcal{S}_{\alpha,\beta}|\le 
\begin{cases}
3&\textrm{when the permutation group induced by $G_\alpha$ on $\beta^{G_\alpha}$ or by $G_\beta$ on $\alpha^{G_\beta}$}\\
&\textrm{ is not an elementary abelian $2$-group of order $4$},\\
9&\textrm{otherwise}.
\end{cases}
\end{equation}
This claim follows from the paragraphs above and from Figure~\ref{figureeye-2}. Indeed, from Lemma~\ref{lemma:-4} part~\eqref{eq:lemma-43}, for each $X\in \mathcal{S}_{\alpha,\beta}$, there exists a proper subgroup $A$ of $G_\alpha$ and a proper subgroup $B$ of $G_\beta$ with $G_\alpha\cap G_\beta<A$, $G_\alpha\cap G_\beta<B$ and $X=AB$. Observe that we have at most $3$ choices for $A$ and at most $3$ choices for $B$ and hence at most $9$ choices for $X$. Moreover, as long as the permutation group induced on the corresponding orbit is not elementary abelian, we actually have only one choice for either $A$ or $B$ yielding at most $3$ choices for $X$.
\smallskip

For each $X\in\mathcal{S}_{\alpha,\beta}$, let $\mathcal{S}_X:=\{\omega\in \Omega_{\alpha,2}\cap\Omega_{\beta,2}\mid G_\omega=X\}$. From Section~\ref{notation} and from the notation therein, we have $|\mathcal{S}_X|=|\Omega_{\omega,1}|=d$. From this and from the definition of $\mathcal{S}_{\alpha,\beta}$, we obtain
\begin{equation}\label{eye:11}
|\Omega_{\alpha,2}\cap\Omega_{\beta,2}|=\left|\bigcup_{X\in\mathcal{S}_{\alpha,\beta}}\mathcal{S}_X\right|= \sum_{X\in\mathcal{S}_{\alpha,\beta}}|\mathcal{S}_X|= |\mathcal{S}_{\alpha,\beta}|d.
\end{equation}

From part~\eqref{eq:lemma-41} of Lemma~\ref{lemma:-4}, we have $\Omega=\Omega_{\alpha,1}\cup\Omega_{\alpha,2}\cup\Omega_{\alpha,4}$. From this, we immediately  get $\Omega_{\beta,2}\subseteq \Omega\setminus\Omega_{\alpha,1}=\Omega_{\alpha,2}\cup\Omega_{\alpha,4}$ and hence
 $\Omega_{\alpha,2}\cup\Omega_{\beta,2}\subseteq \Omega_{\alpha,2}\cup\Omega_{\alpha,4}$. Therefore,
\begin{align}\label{eye:10}
|\Omega_{\alpha,2}\cap\Omega_{\beta,2}|&=|\Omega_{\alpha,2}|+|\Omega_{\beta,2}|-|\Omega_{\alpha,2}\cup\Omega_{\beta,2}|\\\nonumber
&\ge |\Omega_{\alpha,2}|+|\Omega_{\beta,2}|-|\Omega_{\alpha,2}\cup\Omega_{\alpha,4}|\\\nonumber
&= |\Omega_{\alpha,2}|+|\Omega_{\beta,2}|-|\Omega_{\alpha,2}|-|\Omega_{\alpha,4}|\\\nonumber
&=|\Omega_{\alpha,2}|-|\Omega_{\alpha,4}|.
\end{align}
Now, dividing both sides of~\eqref{eye:11} and~\eqref{eye:10} by $|\Omega_{\alpha,1}|=d$, by recalling~\eqref{def:xi} and by rearranging the terms, we obtain 
\begin{equation}\label{eye:30}
x_2\le |\mathcal{S}_{\alpha,\beta}|+x_4.
\end{equation}

\smallskip

We now suppose that part~\eqref{lemma:-30} does not hold  and we show that part~\eqref{lemma:-322},~\eqref{lemma:-32},~\eqref{lemma:-33} and~\eqref{lemma:-34} are satisfied. In particular, we work under the assumption that
$$|\Omega_{\alpha,1}|+|\Omega_{\alpha,2}|\ge \frac{5|\Omega|}{6}.$$
As $|\Omega|=d(x_1+x_2+x_4)$, $|\Omega_{\alpha,1}|+|\Omega_{\alpha,2}|=d(x_1+x_2)$ and $x_1=1$, the inequality $|\Omega_{\alpha,1}|+|\Omega_{\alpha,2}|\ge 5|\Omega|/6$ gives
$$5x_4\le 1+x_2.$$
Now,~\eqref{eye:30} yields $5x_4\le 1+x_2\le 1+|\mathcal{S}_{\alpha,\beta}|+x_4$, that is, $4x_4\le 1+|\mathcal{S}_{\alpha,\beta}|$. From~\eqref{eye:6}, we deduce that $x_4\le 2$.  This already shows part~\eqref{lemma:-322}. 

When $x_4=2$, we deduce $|\mathcal{S}_{\alpha,\beta}|\ge 7$ and hence~\eqref{eye:6} yields that the permutation groups induced by $G_\alpha$ on $\beta^{G_\alpha}$ and by $G_\beta$ on $\alpha^{G_\beta}$ are both elementary abelian $2$-groups of order $4$. Since this argument does not depend upon $\beta\in\Omega_{\alpha,4}$, we have shown that $G_\alpha$ acts as an elementary abelian group on each of its orbits of cadinality $4$. Since all other orbits of $G_\alpha$ have cardinality $1$ or $2$, we deduce that $G_\alpha$ acts as an elementary abelian $2$-group on each of its orbits and hence $G_\alpha$ is an elementary abelian $2$-group.  This shows part~\eqref{lemma:-32}, under the additional assumption that $x_4=2$. Moreover, as $|\mathcal{S}_{\alpha,\beta}|\ge 7$, Lemma~\ref{lemma:44esteso} applied with $X:=G_\alpha$ and $Y:=G_\beta$ gives that $G_\alpha$ and $G_\beta$ commute with each other. This shows that part~\eqref{lemma:-33} is satisfied. To prove part~\eqref{lemma:-34} we use Lemma~\ref{lemma:444esteso}. Let $\Gamma$ be the graph having vertex set $V$, the set of conjugates of $G_\alpha$ in $G$, that is, $$V:=\{G_\omega\mid\omega\in \Omega\}.$$ Then $|V|=1+x_2+x_4$. We declare two vertices $G_{\omega_1}$ and $G_{\omega_2}$ of $\Gamma$ adjacent if $G_{\omega_1}\cap G_{\omega_2}$ has index $4$ in $G_{\omega_1}$ (and hence also in $G_{\omega_2}$). Clearly, the action of $G$ by conjugation gives rise to a vertex-transitive action of $G$ on $\Gamma$. As $x_4=2$, $\Gamma$ has valency $2$. Let $W$ be the collection of all vertices $G_\omega$ of $\Gamma$ with $G_\alpha \cap G_\beta\le G_\omega$. Clearly, $G_\alpha,G_\beta\in W$ and, from  Lemma~\ref{lemma:-4} part~\eqref{eq:lemma-43}, for any two distinct vertices $G_{\delta_1}$ and $G_{\delta_2}$ of $\Gamma$ contained in $W$, we have that $$\Omega_{\delta_1,2}\cap\Omega_{\delta_2,2}=V\setminus (\Gamma(G_{\delta_1})\cup\Gamma(G_{\delta_2}))\subseteq W.$$ From this, Lemma~\ref{lemma:444esteso} gives that either $W=V$ or $|V|\le 6$. The second alternative gives $x_2=|V|-1-x_4\le 3$, which contradicts the fact that $5x_4\le 1+x_2$. Therefore, $W=V$ and hence $G_\alpha\cap G_\beta\le G_\omega$, for every $\omega\in \Omega$. Thus $G_\alpha\cap G_\beta=1$ and hence $G_\alpha G_\beta=G_\alpha\times G_\beta$ is an elementary abelian $2$-group of order $16$. To prove that $G_\alpha\times G_\beta\unlhd G$ it suffices to apply again this argument to the collection $W$ of all vertices $G_\omega$ of $\Gamma$ with $G_\omega\le G_\alpha \times G_\beta$.

In particular, in the rest of the proof we work under the assumption $x_4=1$.

When $x_4=1$, we may refine some of the inequalities above. Indeed, when $x_4=1$, we have $\Omega_{\alpha,4}=\Omega_{\beta,1}$, because both sets have the same cardinality and $\Omega_{\beta,1}\subseteq \Omega_{\alpha,4}$. From this it follows $\Omega_{\alpha,2}=\Omega_{\beta,2}$. Therefore, from~\eqref{eye:11}, we get
$$dx_2=|\Omega_{\alpha,2}|=|\Omega_{\alpha,2}\cap\Omega_{\beta,2}|=d|\mathcal{S}_{\alpha,\beta}|.$$
Now, the inequality $5=5x_4\le 1+x_2$ implies $|\mathcal{S}_{\alpha,\beta}|=x_2\ge 4$. Again, we may use~\eqref{eye:6} to deduce that the permutation groups induced by $G_\alpha$ on $\beta^{G_\alpha}$ and by $G_\beta$ on $\alpha^{G_\beta}$ are both elementary abelian $2$-groups of order $4$. 
This, as above, yields that $G_\alpha$ is an elementary abelian $2$-group, that is, part~\eqref{lemma:-32} holds.

From Lemma~\ref{lemma:-4} part~\eqref{eq:lemma-43}, $G_\alpha\cap G_\beta\le G_\omega$, for every $\omega\in \Omega_{\alpha,2}\cap \Omega_{\beta,2}$. In particular, $G_\alpha\cap G_\beta$ fixes pointwise $\Omega_{\alpha,2}\cap\Omega_{\beta,2}$. As $\Omega_{\alpha,2}\cap \Omega_{\beta,2}=\Omega_{\alpha,2}$, we deduce that $G_\alpha\cap G_\beta$ fixes pointwise $\Omega_{\alpha,2}$. Since $G_\alpha\cap G_\beta$ fixes pointwise also $\Omega_{\alpha,1}$ and $\Omega_{\beta,1}=\Omega_{\alpha,4}$, we obtain that $G_\alpha\cap G_\beta$ fixes pointwise $\Omega_{\alpha,1}\cup\Omega_{\alpha,2}\cup\Omega_{\alpha,4}=\Omega$. Thus $G_\alpha\cap G_\beta=1$ and $|G_\alpha|=4$. Observe also that when $x_4=1$, the hypothesis of Lemma~\ref{lemma:4esteso} are satisfied and hence ${\bf N}_G(G_\alpha)={\bf N}_G(G_\beta)$. Therefore $G_\beta$ normalizes $G_\alpha$. This gives that the commutator subgroup $[G_\alpha,G_\beta]$ lies in $G_\alpha\cap G_\beta=1$, that is, $G_\alpha$ commutes with $G_\beta$. This shows that part~\eqref{lemma:-33} is satisfied. Now, as $\Omega_{\alpha,2}=\Omega_{\beta,2}$,  Lemma~\ref{lemma:-4} part~\eqref{eq:lemma-43} yields $G_\omega\le G_\alpha\times G_\beta$, for every $\omega\in \Omega_{\alpha,2}$. Therefore, $G_{\alpha}\times G_\beta$ contains $G_\omega$, for every $\omega\in \Omega$. Thus $$G_\alpha\times G_\beta=\langle G_\omega\mid\omega\in \Omega\rangle\unlhd G$$
and $G_\alpha\times G_\beta$ has order $16$. Thus part~\eqref{lemma:-34} is satisfied.
\end{proof}

We need one final preliminary lemma, with a somehow different flavour. We denote by $C_2$ and $C_4$ the cyclic groups of order $2$ and $4$, respectively, we denote by $Q_8$ the quaternion group of order $8$ and we denote by $D_8$ the dihedral group of order $4$.
\begin{lemma}\label{lemma:diff}
Let $R$ be a finite group, let $U$ be a proper subgroup of $R$ and let $r\in U$ be a central involution of $R$. Let $\tau:R\to R$ be the permutation defined by
$$
x\mapsto x^\tau:=
\begin{cases}
x&\textrm{when }x\in U,\\
xr&\textrm{when }x\in R\setminus U.
\end{cases}
$$
Then one of the following holds
\begin{enumerate}[(a)]
\item\label{eq:diff1}the number of inverse-closed subsets $S$ of $R$ with $S^\tau=S$ is at most $2^{{\bf c}(R)-\frac{|R|}{48}}$,
\item\label{eq:diff2}$R$ is generalized dicyclic,
\item\label{eq:diff22}$R\cong C_4\times C_2^\ell$, for some non-negative integer $\ell$.
\end{enumerate}
\end{lemma}
\begin{proof}
Let $\iota:R\to R$ be the permutation defined by $x^\iota=x^{-1}$, for every $x\in R$, and let $T:=\langle\iota,\tau\rangle$. Observe that a subset $S$ of $R$ is inverse-closed and $\tau$-invariant if and only if $S$ is $T$-invariant. In particular, the number of inverse-closed subsets $S$ of $R$ with $S^\tau=S$ is $2^\kappa$, where $\kappa$ is the number of orbits of $T$ on $R$. To compute $\kappa$ we use the orbit-counting lemma, which says that
\begin{equation}\label{ocl}
\kappa=\frac{1}{|T|}\sum_{t\in T}|\mathrm{Fix}_R(t)|.
\end{equation}

Observe that
\begin{align}\label{anuw}
\mathrm{Fix}_R(1)&:=R,\\\nonumber
\mathrm{Fix}_R(\iota)&:={\bf I}(R),\\\nonumber
\mathrm{Fix}_R(\tau)&:=U,\\\nonumber
\mathrm{Fix}_R(\iota\tau)&:={\bf I}(U)\cup\{x\in R\setminus U\mid x^2=r\}.\nonumber
\end{align}

Observe that $\iota\tau=\tau\iota$ and $\iota^2=\tau^2=1$. Therefore $T$ is an elementary abelian $2$-group of order at most $4$.

Observe that $\tau\ne 1$, because $U$ is a proper subgroup of $R$ and $r\ne 1$. If $\iota=1$, then $R$ is an elementary abelian $2$-group and $T=\langle \tau\rangle$. Thus~\eqref{ocl} and~\eqref{anuw} yield
\begin{align*}
\kappa&=\frac{1}{2}\left(|R|+|U|\right)\le\frac{|R|}{2}+\frac{|R|}{4}= \frac{3|R|}{4}\\
&=|R|-\frac{|R|}{4}={\bf c}(R)-\frac{|R|}{4}.
\end{align*}
Therefore, part~\eqref{eq:diff1} holds and the proof follows in this case. Suppose now $\iota=\tau$. This means that $U$ is an elementary abelian $2$-subgroup of $R$ and $x^{-1}=xr$, for every $x\in R\setminus U$. In other words, all elements in $U$ square to $1$ and all elements in $R\setminus U$ square to $r$. Let $\bar{R}:=R/\langle r\rangle$ and let us use the ``bar'' notation for the subgroups and for the elements of $\bar{R}$. Consider the function
$$( \cdot,\cdot):\bar{R}\times\bar{R}\to \langle r\rangle$$
defined by $( x\langle r\rangle, y\langle r\rangle)=x^{-1}y^{-1}xy$, for every $x,y\in R$. Similarly, consider the function
$$q:\bar{R}\to\langle r\rangle$$
defined by $q(x\langle r\rangle)=x^2$. It is not hard to see that, regarding $\bar{R}$ as a vector space over the field with $2$ elements, $(\cdot,\cdot)$ is a bilinear form and $q$ is a quadratic form polarizing to $(\cdot,\cdot)$, that is, $$q(\bar{x}\bar{y})q(\bar{x})q(\bar{y})=(\bar{x},\bar{y}),$$
for every $\bar{x},\bar{y}\in \bar{R}$.
Using this terminology, we have that each element of $\bar{U}$ is totally singular and each element of $\bar{R}\setminus\bar{U}$ is non-degenerate. From the classification of the quadratic forms over finite fields, we have $|\bar{R}:\bar{U}|\in \{2,4\}$. When $|\bar{R}:\bar{U}|=2$, we deduce that 
$R$ is an abelian group isomorphic to the direct product $C_4\times C_2^\ell$, for some $\ell\ge0$. In particular, part~\eqref{eq:diff22} holds. When $|\bar{R}:\bar{U}|=4$, we deduce that 
$R\cong Q_8\times C_2^\ell$, for some $\ell\ge0$. In particular, $R$ is generalized dicyclic and part~\eqref{eq:diff2} holds. For the rest our our argument, we may suppose that 
$\tau\ne\iota\ne1$.

The paragraph above can be summarized by saying that $T=\langle\iota,\tau\rangle$ has order $4$ and hence, from~\eqref{anuw},~\eqref{ocl} becomes
\begin{align}\label{ocl1}
\kappa&=
\frac{1}{4}
\left(
|\mathrm{Fix}_R(1)||+
|\mathrm{Fix}_R(\iota)|+
|\mathrm{Fix}_R(\tau)|+
|\mathrm{Fix}_R(\tau\iota)|
\right)\\\nonumber
&=\frac{1}{4}\left(
|R|+|{\bf I}(R)|+|U|+|{\bf I}(U)|+|\{x\in R\setminus U\mid x^2=r\}|\right)\\\nonumber
&\le 
\frac{1}{4}\left(
|R|+|{\bf I}(R)|+|U|+|{\bf I}(R)|+|\{x\in R\setminus U\mid x^2=r\}|\right)\\\nonumber
&=\frac{|R|+|{\bf I}(R)|}{2}-\left(
\frac{|R|}{4}-\frac{|U|}{4}-\frac{|\{x\in R\setminus U\mid x^2=r\}|}{4}
\right)\\\nonumber
&={\bf c}(R)-\left(
\frac{|R|}{4}-\frac{|U|}{4}-\frac{|\{x\in R\setminus U\mid x^2=r\}|}{4}
\right).\nonumber
\end{align}
Set $\mathcal{S}:=\{x\in R\setminus U\mid x^2=r\}$.

If $\mathcal{S}=\emptyset$, then the proof follows immediately from~\eqref{ocl1}, indeed, part~\eqref{eq:diff1} holds true. Therefore, for the rest of the proof we suppose $$\mathcal{S}\ne\emptyset.$$

To conclude we divide the proof in various cases. 

Suppose first that $|R:U|=2$. Let $x\in \mathcal{S}$ and observe that $R=U\cup Ux$. Now, a computation yields
$$\mathcal{S}=\{ux\mid u\in U, u^x=u^{-1}\}.$$
When $\mathcal{S}=Ux$, the action of $x$ on $U$  by conjugation is an automorphism of $U$ inverting each element of $U$. Therefore $U$ is abelian and $R$ is generalized dicyclic. Hence part~\eqref{eq:diff2} holds. When $\mathcal{S}\subsetneq Ux$, the result of Liebeck and MacHale~\cite{LieMac} shows that the automorphism $x$ can invert at most $3/4$ of the elements of $U$ and hence $|\mathcal{S}|\le 3|U|/4=3|R|/8$. Now,~\eqref{ocl1} gives $\kappa\le {\bf c}(R)-|R|/32$; hence part~\eqref{eq:diff1} holds and the proof follows. Therefore, for the rest of the proof we may suppose
\begin{equation}\label{bound}|R:U|\ge3.
\end{equation}

When $|\mathcal{S}|\le 3|R|/4-|U|/2$, from~\eqref{ocl1} and~\eqref{bound}, we deduce 
\begin{align*}
\kappa&\le {\bf c}(R)-\left(
\frac{|R|}{4}-\frac{|U|}{4}-\frac{3|R|}{16}+\frac{|U|}{8}\right)\\
&= {\bf c}(R)-\left(
\frac{|R|}{16}-\frac{|U|}{8}\right)\\
&\le{\bf c}(R)-\left(
\frac{|R|}{16}-\frac{|R|}{24}\right)={\bf c}(R)-
\frac{|R|}{48}
\end{align*} and the proof follows. Therefore, for the rest of the proof, we suppose $$|\mathcal{S}|> 3|R|/4-|U|/2.$$

Let $u$ be an arbitrary element of $U$. Then $u\mathcal{S}\subseteq R\setminus U$ and hence $\mathcal{S}\cup u\mathcal{S}\subseteq R\setminus U$. Therefore
\begin{align}\label{ocl2}
|\mathcal{S}\cap u\mathcal{S}|&=
|\mathcal{S}|+|u\mathcal{S}|-|\mathcal{S}\cup u\mathcal{S}|=2|\mathcal{S}|-|\mathcal{S}\cup u\mathcal{S}|\\\nonumber
&\ge 2|\mathcal{S}|-(|R|-|U|)>\frac{3|R|}{2}-|U|-(|R|-|U|)=\frac{|R|}{2}.\nonumber
\end{align}
Now, let $ux\in\mathcal{S}\cap u\mathcal{S}$. Then $x\in\mathcal{S}$ and hence 
$$r=(ux)^2=uxux=uu^xx^2=uu^xr.$$
Therefore $u^x=u^{-1}$. Now, repeating the argument above with $y\in \mathcal{S}\cap u\mathcal{S}$, we deduce $u^y=u^{-1}$ and hence $xy^{-1}\in {\bf C}_R(u)$. Since we have $|\mathcal{S}\cap u\mathcal{S}|$ choices for $y$,~\eqref{ocl2} implies $|{\bf C}_R(u)|>|R|/2$ and hence $R={\bf C}_R(u)$. Since $u$ is an arbitrary element of $U$, we deduce that $U$ is a central subgroup of $R$.

Since $u^x=u^{-1}$, for every $u\in U$ and for every $ux\in \mathcal{S}\cap u\mathcal{S}$, and since $U$ is contained in the center of $R$, we deduce that $U$ has exponent $2$. Since $U$ is a central subgroup of $R$ of exponent $2$, we now have an easier description of for $\mathcal{S}$, that is,
$$\mathcal{S}=\{x\in R\mid x^2=r\}.$$

Now that we know that $U$ has exponent $2$, we consider the quotient group $\bar{R}:=R/\langle r\rangle$. Observe that each element of $\bar U$ is an involution. Assume that $\bar{R}$ is not an elementary abelian $2$-group. Then, the theorem of Miller~\cite{Miller} yields $|{\bf I}(\bar{R})|\le 3|\bar{R}|/4$. In particular, the number of involutions in $\bar{R}\setminus \bar{U}$ is at most $3|\bar{R}|/4-|\bar{U}|$. Since each element in $\bar{\mathcal{S}}$ is an involution and since $\bar{\mathcal{S}}\subseteq \bar{R}\setminus \bar{U}$, we deduce $|\mathcal{S}|\le 3|R|/4-|U|$. Using this inequality  in~\eqref{ocl1}, we get
$$\kappa\le {\bf c}(R)-\frac{|R|}{16},$$
part~\eqref{eq:diff1} holds
and the proof follows in this case. It remains to consider the case that $\bar{R}$ is an elementary abelian $2$-group.

For this remaining case, we consider the bilinear form
$$( \cdot,\cdot):\bar{R}\times\bar{R}\to \langle r\rangle$$
defined by $( x\langle r\rangle, y\langle r\rangle)=x^{-1}y^{-1}xy$, for every $x,y\in R$, and its quadratic form
$$q:\bar{R}\to\langle r\rangle$$
defined by $q(x\langle r\rangle)=x^2$.  Again,  we use the classification of the quadratic forms over finite fields. Using this terminology, $\bar{U}$ is totally isotropic and contained in the kernel of the the bilinear form $(\cdot,\cdot)$ and the elements of $\bar{\mathcal{S}}$ are non-degenerate. Using this information we obtain that $R$ is isomorphic to one of the following groups
\begin{itemize}
\item $C_4\times C_2^\ell$, for some $\ell\ge 0$,
\item $\underbrace{D_8\circ D_8\circ \cdots \circ D_8}_{t\textrm{ times}}\times C_2^\ell$, for some $\ell\ge 0$ and $t\ge 1$,
\item $Q_8\circ\underbrace{D_8\circ D_8\circ \cdots \circ D_8}_{(t-1)\textrm{ times}}\times C_2^\ell$, for some $\ell\ge 0$ and some $t\ge 1$,
\item $C_4\circ \underbrace{D_8\circ D_8\circ \cdots \circ D_8}_{t\textrm{ times}}\times C_2^\ell$, for some $\ell\ge 0$ and some $t\ge 1$.
\end{itemize}
In the first case, an explicit computation gives $|\mathcal{S}|=|R|/2$. Hence~\eqref{ocl1} gives
\begin{align*}
\kappa&\le{\bf c}(R)-\left(\frac{|R|}{4}-\frac{|U|}{4}-\frac{|R|}{8}\right)=
{\bf c}(R)-\left(\frac{|R|}{8}-\frac{|U|}{4}\right)\\
&\le{\bf c}(R)-\left(\frac{|R|}{8}-\frac{|R|}{16}\right)
={\bf c}(R)-\frac{|R|}{16}
\end{align*}
and part~\eqref{eq:diff1} holds. In the second case, an explicit computation gives $|\mathcal{S}|=(2^t-1)|R|/2^{t+1}\le |R|/2$. Therefore we may argue as in the previous case and we obtain that part~\eqref{eq:diff1} holds. In the third case, an explicit computation gives $|\mathcal{S}|=(2^t+1)|R|/2^{t+1}$. When $t=1$, $R\cong Q_8\times C_2^\ell$ is generalized dicyclic and hence part~\eqref{eq:diff2} hods. When $t\ge 2$, we have  $|\mathcal{S}|\le 5|R|/8$ and hence~\eqref{ocl1} gives
\begin{align*}
\kappa&\le{\bf c}(R)-\left(\frac{|R|}{4}-\frac{|U|}{4}-\frac{5|R|}{32}\right)=
{\bf c}(R)-\left(\frac{3|R|}{32}-\frac{|U|}{4}\right)\\
&\le{\bf c}(R)-\left(\frac{3|R|}{32}-\frac{|R|}{16}\right)
={\bf c}(R)-\frac{|R|}{32}.
\end{align*} 
Thus, we obtain that part~\eqref{eq:diff1} holds. In the forth (and last) case, an explicit computation gives $|\mathcal{S}|=|R|/2$. Therefore we may argue as in the first case and we obtain that part~\eqref{eq:diff1} holds.
\end{proof}

\section{Proof of Theorems~\ref{thrm:main1} and~\ref{thrm:main2}}
In this section, using Section~\ref{sec:lemmata} we prove both Theorems~\ref{thrm:main1} and~\ref{thrm:main2}. Thus, let $G$ be a finite transitive permutation group on $\Omega$ with $${\bf I}_\Omega(G)\ge \frac{5}{6}.$$
If ${\bf I}_\Omega(G)=1$, then there is nothing to prove and hence we may suppose that ${\bf I}_\Omega(G)<1$. Let $\alpha\in \Omega$. From Lemma~\ref{lemma:-4}, we have
$$\Omega=\Omega_{\alpha,1}\cup\Omega_{\alpha,2}\cup\Omega_{\alpha,4}.$$ 
Since ${\bf I}_\Omega(G)<1$, $\Omega_{\alpha,4}\ne\emptyset$. Let $\beta\in \Omega_{\alpha,4}$. From Lemma~\ref{lemma:-3},
$$V:=G_\alpha\times G_\beta$$
is an elementary abelian normal $2$-subgroup of $G$ of order $16$. Let $e_1,e_2,e_3,e_4$ be a basis of $V$, regarded as a vector space over the field with $2$ elements, and with $G_\alpha=\langle e_1,e_2\rangle$. Let $H:=G/{\bf C}_G(V)$ and $W:=G_\alpha$. Clearly, $H\le \mathrm{GL}(V)\cong\mathrm{GL}_4(2)$. Now, consider the action of $H$ on the $2$-dimensional subspaces of $V$ and consider $O:=\{W^h\mid h\in H\}$, the $H$-orbit containing $W$. Clearly,
$$\frac{|\Omega_{\alpha,1}\cup\Omega_{\alpha,2}|}{|\Omega|}=\frac{|\{U\in O\mid |W:W\cap U|\le 2\}|}{|O|}.$$
Observe that the right hand side of this equality can be easily computed with the help of a computer. With the computer algebra system~\texttt{magma}~\cite{magma}, we have computed all the subgroups of $\mathrm{GL}_4(2)$. Then, we have selected only the subgroups $H$ with the property that
$$V=\langle W^h\mid h\in H\rangle\hbox{   and   }\bigcap_{h\in H}W^h=0.$$
(This selection is due to the fact that $V=\langle G_\alpha^g\mid g\in G\rangle$ and that $G_\alpha$ is core-free in $G$.)
 Then, for each such subgroup $H$, we have computed the orbit $O=W^H$ and we have computed the ratio $\frac{|\{U\in O\mid |W:W\cap U|\le 2\}|}{|O|}$. We have checked that in all cases this ratio is at most $5/6$. In particular, Theorem~\ref{thrm:main1} is proved. Moreover, we have checked that this ratio is $5/6$ if and only if $H$ is given in the statement of Theorem~\ref{thrm:main2}. Since this construction can be reversed, we also obtain the converse implication for Theorem~\ref{thrm:main2}.

\section{Proof of Theorem~\ref{thrm:main3}}

Let $G$ be a finite transitive group properly containing a regular subgroup $R$. Since $R$ acts regularly, we may identify the domain of $G$ with $R$. Now, the number of Cayley graphs $\Cay(R,S)$ on $R$ with $G\le \mathrm{Aut}(\Cay(R,S))$ is the number of inverse-closed subsets $S$ of $R$ left invariant by $G_1$, where $G_1$ is the stabilizer of the point $1\in R$ in $G$. In particular, to prove Theorem~\ref{thrm:main3}, we need to estimate the number of inverse-closed subsets of $R$ that are union of $G_1$-orbits.

Suppose first that $${\bf I}_R(G)=1.$$ Since $R$ is properly contained in $G$, from the theorem of Bergman and Lenstra mentioned in Section~\ref{intro}, we have two cases to consider
\begin{itemize}
\item $|G_1|=2$,
\item $G$ contains an elementary abelian normal $2$-subgroup $N$ with $|N:G_1|=2$.
\end{itemize}
Assume first that $|G_1|=2$. Let $\varphi\in G_1\setminus\{1\}$. From the Frattini argument, $G=RG_1$ and hence $|G:R|=2$. This gives $R\unlhd G$ and hence $\varphi$ acts by conjugation on $R$ as a group automorphism. Now, from~\cite[Lemma~$2.7$]{Bxu} or~\cite[Theorem~1.13]{MMS}, we have that  
\begin{enumerate}[(a)]
\item the number of $\varphi$-invariant inverse-closed subsets of $R$ is at most $2^{{\bf c}(R)-\frac{|R|}{96}}$, or
\item $R$ is abelian of exponent greater than $2$ and $\varphi$ is the automorphism of $R$ mapping each element to its inverse, or
\item $R$ is generalized dicyclic and $\varphi$ is an automorphism of $R$ with $x^\varphi\in \{x,x^{-1}\}$, for every $x\in R$.
\end{enumerate}
In particular, the proof of Theorem~\ref{thrm:main3} follows in this case.

Assume next that $G$ contains an elementary abelian normal $2$-subgroup $N$ with $|N:G_1|=2$. Since $R$ acts transitively, $G=RN$. Moreover, since $R$ acts regularly, $G=RG_1$ and $R\cap G_1=1$. Thus $|R\cap N|=|N|/|G_1|=2$. Let $r$ be a generator of $R\cap N$. Since $\langle r\rangle=R\cap N\unlhd R$, $r$ is a central involution of $R$. Let $U:={\bf N}_R(G_1)$. Since ${\bf N}_R(G_1)$ is a block of imprimitivity for $G$, $U={\bf N}_R(G_1)$ is also a block of imprimitivity for the regular action of $R$ and hence $U$ is a subgroup of $R$. As $G_1\ne1 $ because $R$ is properly contained in $G$, we deduce that $U$ is a proper subgroup of $R$. Now, $G_1$ fixes pointwise $U$ and, for every $x\in R\setminus U$, we have
$$x^{G_1}=\{x,xr\}.$$
Let $\tau:R\to R$ be the permutation defined by
$$
x\mapsto x^\tau:=
\begin{cases}
x&\textrm{when }x\in U,\\
xr&\textrm{when }x\in R\setminus U.
\end{cases}
$$
We have shown that $S\subseteq R$ is $G_1$-invariant if and only if $S$ is $\langle\tau\rangle$-invariant. Therefore, the proof of this case follows from Lemma~\ref{lemma:diff}.

To conclude the proof of Theorem~\ref{thrm:main3}, it remains to consider the case that $${\bf I}_R(G)\ne 1.$$ From Theorem~\ref{thrm:main1}, we have ${\bf I}_R(G)\le 5/6$. Recall that ${\bf I}(R)=\{x\in R\mid x^2=1\}.$ We define
\begin{align*}
&a:=|\Omega_{R,1}\cap {\bf I}(R)|,&&b:=|\Omega_{R,1}\cap (R\setminus {\bf I}(R))|,\\
&c:=|\Omega_{R,2}\cap {\bf I}(R)|,&&d:=|\Omega_{R,2}\cap (R\setminus {\bf I}(R))|,\\
e&:=|(R\setminus (\Omega_{R,1}\cup \Omega_{R,2}))\cap {\bf I}(R)|,&&f:=|(R\setminus (\Omega_{R,1}\cup \Omega_{R,2}))\cap (R\setminus {\bf I}(R))|.
\end{align*}
As ${\bf I}_R(G)\le 5/6$, we deduce
\begin{equation}\label{final?}
\frac{|R|}{6}\le|R\setminus(\Omega_{R,1}\cup\Omega_{R,2})|=e+f.
\end{equation}

Let $\iota:R\to R$ be the permutation defined by $x^\iota:=x^{-1}$, for every $x\in R$, and  let $T:=\langle \iota,G_1\rangle$. Now, the number of $G_1$-invariant inverse-closed subsets of $R$ is exactly the number of $T$-invariant subsets of $R$. Moreover, the number of $T$-invariant subsets of $R$ is $2^\kappa$, where $\kappa$ is the number of orbits of $T$ on $R$.

The group $T$ has
\begin{itemize}
\item orbits of cardinality $1$ on $\Omega_{R,1}\cap {\bf I}(R)$,
\item orbits of cardinality $2$ on $\Omega_{R,1}\cap (R\setminus {\bf I}(R))$,
\item orbits of cardinality $2$ on $\Omega_{R,2}\cap {\bf I}(R)$,
\item orbits of cardinality at least $2$ on $\Omega_{R,2}\cap (R\setminus {\bf I}(R))$,
\item orbits of cardinality at least  $3$ on $(R\setminus(\Omega_{R,1}\cup\Omega_{R,2}))\cap {\bf I}(R)$,
\item orbits of cardinality at least $4$ on $(R\setminus(\Omega_{R,1}\cup\Omega_{R,2}))\cap (R\setminus {\bf I}(R))$.
\end{itemize}
All of these assertions are trivial except, possibly, the last one. Indeed, if $x\in (R\setminus(\Omega_{R,1}\cup\Omega_{R,2}))\cap (R\setminus {\bf I}(R))$, then $x$ is not an involution and the $G_1$-orbit $x^{G_1}$ has cardinality at least $3$. As
$$(x^{G_1})^{-1}=(x^{-1})^{G_1},$$
we deduce that $|x^T|$ has even cardinality and hence $|x^T|$ is at least $4$.

Summing up, we have
\begin{align*}
\kappa&\le a+\frac{b}{2}+\frac{c}{2}+\frac{d}{2}+\frac{e}{3}+\frac{f}{4}=a+c+e+\frac{b}{2}+\frac{d}{2}+\frac{f}{2}-\left(\frac{c}{2}+\frac{2e}{3}+\frac{f}{4}\right)\\
&=\frac{|R|+|{\bf I}(R)|}{2}-\left(\frac{c}{2}+\frac{2e}{3}+\frac{f}{4}\right)
={\bf c}(R)-\left(\frac{c}{2}+\frac{2e}{3}+\frac{f}{4}\right)\\
&\le{\bf c}(R)-\left(\frac{2e}{3}+\frac{f}{4}\right)\le {\bf c}(R)-\left(\frac{e}{4}+\frac{f}{4}\right)\\
&\le{\bf c}(R)-\frac{|R|}{24},
\end{align*}
where in the last inequality we have used~\eqref{final?}. This concludes the proof of Theorem~\ref{thrm:main3}.

\end{document}